\documentclass[12 pt, a4paper xcolor=dvipsnames]{amsart}

\usepackage{imakeidx}
\makeindex[columns=3, title=Alphabetical Index]

\usepackage{amssymb, amsmath, amsthm, mathtools}

\usepackage{color}

\newtheorem{lemma}{Lemma}[section]
\newtheorem{proposition}[lemma]{Proposition}

\newtheorem{claim}[lemma]{Claim}

\newtheorem{conjecture}[lemma]{Conjecture}
\newtheorem{corollary}[lemma]{Corollary}

  \newcommand{\Label}[1]{\label{#1}\textcolor{green}{\tiny #1} }
  \renewcommand{\Label}[1]{\label{#1}}

\newcommand{\PP}{\mathsf{P}}
\newcommand{\ML}{\mathsf{ML}}
\newcommand{\TT}{\mathsf{T}}
\renewcommand{\H}{\mathbb{H}}
\newcommand{\CP}{\mathbb{C}{\rm P}}

\newcommand{\s}{\mathbb{S}}

\newcommand{\R}{\mathbb{R}}

\newcommand{\C}{\mathbb{C}}
\newcommand{\N}{\mathbb{N}}
\newcommand{\PSL}{{\rm PSL}}

\newcommand{\LLL}{\mathcal{L}}
\newcommand{\Conv}{\operatorname{Conv}}
\newcommand{\Isom}{\operatorname{Isom}}

\newcommand{\cc}{\circ}
\newcommand{\infi}{\infty}
\newcommand{\bdr}{\partial}

\newcommand{\alp}{\alpha}
\newcommand{\kap}{\kappa}
\newcommand{\gam}{\gamma}
\newcommand{\Gam}{\Gamma}

\newcommand{\ti}[1]{\tilde{#1}}

\renewcommand{\Im}{\rm Im}
\newcommand{\Core}{\rm Core}
\newcommand{\Gr}{\rm Gr}

\newcommand{\Lam}{\Lambda}

\newcommand{\minus}{\setminus}
\newcommand{\col}{\colon}
\newcommand{\til}{\tilde}
\newcommand{\ol}{\overline}
\newcommand{\sub}{\subset}

\newcommand{\Qed}[1]{\nopagebreak[4]{\tiny \hfill
\fbox{\ref{#1}} \linebreak }\pagebreak[2]}

\definecolor{dblue}{cmyk}{1,0, 0,.7}
\usepackage{hyperref}

\newtheorem{thm}{Theorem}

\begin{document}

\markright{ Thurston}

\title[\today]{On Thurston's parameterization of $\CP^1$-structures}

\author{Shinpei Baba}
\address{Osaka University 
}

\email{baba@math.sci.osaka-u.ac.jp}
\maketitle

\begin{abstract}

Thurston established a correspondence between $\CP^1$-structures  (complex projective structures)  and equivariant pleated surfaces in the hyperbolic-three space $\H^3$,  in order to give a parameterization of the deformation space of $\CP^1$-structures. 
In this note, we summarize Thurston's parametrization of $\CP^1$-structures, based on  \cite{Kamishima-Tan-92} and  \cite{Kullkani-Pinkall-94}, giving an outline and the key points of its construction. 

 In addition we give independent proofs for the following well-known theorems on $\CP^1$-structures by means of pleated surfaces given by the parameterization. 
(1) Goldman's Theorem on $\CP^1$-structures with quasi-Fuchsian holonomy. (2) The path lifting property of developing maps in the domain of discontinuities in $\CP^1$. 

\end{abstract}

\markright{Parametrization of $\CP^1$-structures}

\tableofcontents
\section{Introduction}
Let $\PP$ be the space of all (marked) $\CP^1$-structures on  a closed oriented surface  $S$ of genus at least two (\S \ref{CPOneStructures}). 
Thurston gave the following parameterization of $\PP$, using pleated surfaces in the hyperbolic three-space $\H^3$. 

\begin{thm}(Thurston, \cite{Kamishima-Tan-92}  \cite{Kullkani-Pinkall-94} )\Label{ThurstonCoordimates} 
$$\PP \cong \ML \times \TT,$$
where $\ML$ is the space of measured laminations on $S$ and $\TT$ is the space of all (marked) hyperbolic structures on $S$. 
\end{thm}

In \S \ref{ThurstonParameter}, we outline this correspondence, in part, giving more details, following the work of Kulkarni and Pinkall \cite{Kullkani-Pinkall-94}.  
A hyperbolic structure on $S$ is in particular a $\CP^1$ structure, and its holonomy is a discrete and faithful representation of $\pi_1(S)$ into $\PSL(2, \R)$, called a {\it Fuchsian representation}. 
One holonomy representation of a $\CP^1$-structure on $S$ corresponds to countably many different $\CP^1$-structures on $S$. 
Indeed, there is an operation called $2\pi$-grafting (or simply grafting) which transforms a $\CP^1$-structure to  a new $\CP^1$-structure, preserving its holonomy representations. 
The following theorem of Goldman characterizes all $\CP^1$-structures with fixed Fuchsian holonomy.
\begin{thm}[\cite{Goldman-87}]\Label{Goldman}
Every $\CP^1$-structure $C$  on $S$ with Fuchsian holonomy $\rho$ is obtained by grafting the hyperbolic structure $\tau$ along a unique multiloop $M$. 
\end{thm}
Goldman actually proved the theorem for more general quasi-Fuchsian groups, although the proof is immediately reduced to the case of  Fuchsian representations by a quasiconformal map of $\CP^1$.
Let $C$ be a $\CP^1$-structure with Fuchsian holonomy $\pi_1(S) \to \PSL(2, \C)$. 
Then,  by Theorem \ref{Goldman},   $C$ corresponds to $(\tau, M)$, where $\tau$ is the hyperbolic structure $\H^2/ \Im \rho$ and each loop of $M$ has a $2\pi$-multiple weight. 

For a subgroup $\Gamma \sub \PSL(2, \C)$,  the {\it limit set} of $\Gamma$ is the set of the accumulation points of a $\Gamma$-orbit in $\CP^1$, and the domain of discontinuity is the complement of the limit set in $\CP^1$. 
In \S \ref{ProofGoldman}, we give an alternative proof of Theorem \ref{Goldman}, directly using pleated surfaces given by the Thurston parameters. 

The following Theorem is a technical part of the proof of Theorem \ref{Goldman}, which was originally missing. 
\begin{thm}[\cite{Choi-Lee-97}, see also \S14.4.1. in  \cite{Goldman(18)}]\Label{Covering}
Let $(f, \rho)$ be a developing pair of a $\CP^1$-structure on $S$. 
Let $\Omega$ be the domain of discontinuity of $\Im \rho$.
Then, for each connected component $U$ of $f^{-1}(\Omega)$,  the restriction of $f$ to $U$ is a covering map onto its image.
\end{thm}

Note that as  developing  maps are local homeomorphisms, Theorem \ref{Covering} is equivalent to saying that $f$  has the path lifting property in the domain of discontinuity of $\Im \rho$.

We also give an alternative proof of  Theorem \ref{Covering} in \S\ref{LfitingProperty}, using Thurston's parametrization. 

Theorem \ref{Goldman} states that given two $\CP^1$-structures $C_1$ and $C_2$ with Fuchsian holonomy, $C_1$ can be transformed into $C_2$, via the hyperbolic structure, by a composition of an inverse-grafting and a grafting (where an inverse grafting is the opposite of grafting which remove a cylinder for $2\pi$-grafting). 
 The following question due to Gallo, Kapovich, and Marden  remains open.

\begin{conjecture}[\S 12.1  in \cite{Gallo-Kapovich-Marden}]
Given two $\CP^1$-structures $C_1, C_2$ on $S$ with fixed holonomy $\pi_1(S) \to \PSL(2, \C)$,  there is a composition of  grafts and  inverses of grafts which transforms $C_1$ into $C_2$.
\end{conjecture}

Although \cite{Gallo-Kapovich-Marden} stated this conjecture in the form of a question, we state it more positively since it has been solved affirmatively for generic holonomy representations, namely, for purely loxodromic representations
\cite{ Baba-15gt, Baba-17}.  (For Schottky representations, see \cite{Baba12}.)
There is also a version of this question for branched $\CP^1$-structures (Problem 12:1:2 in \cite{Gallo-Kapovich-Marden}); see
\cite{Calsamiglia4DeroinFrancaviglia14} \cite{Ruffoni20} for some progress in the case of branched $\CP^1$-structures. 

Recently, Gupta and Mj \cite{GuptaMj(19)} gave a generalization of Theorem \ref{ThurstonCoordimates} to certain $\CP^1$-structures on a surface with punctures (namely, $\CP^1$-structures  which corresponds to compact Riemann surfaces with meromorphic quadratic differentials whose poles are of order at least three); see also \cite{AllegrettiBridgeland(18)MonodromoyOfMeromorphicProjectiveStructures}.

{\bf Acknowledgements:}
I thank Gye-Seon Lee and  the mathematics department of Universit\"at Heidelberg  for their hospitality where much of this chapter was written.
The author is  supported by  JSPS Grant-in-Aid for Research Activity start-up (18H05833) and
Grant-in-Aid for Scientific Research C (20K03610).
He also thanks the anonymous referee and the editor for comments, which in particular made some parts of the arguments clearer. 

\section{$\CP^1$-structures on surfaces}\Label{CPOneStructures}  
General references for $\CP^1$-structures are can be found in, for example, \cite{Dumas-08, Kapovich-01}.

A $\CP^1$-{\it  structure} on $S$ is  a $(\CP^1,  \PSL(2, \C))$-structure, i.e. a maximal atlas of charts  embedding open subsets of $S$ onto open subsets \ of $\CP^1$ such that their transition maps are in $\PSL(2, \C)$. 
Let $\ti{S}$ be the universal cover of $S$, which is topologically an open disk. 
Then, equivalently, a $\CP^1$-structure on $S$ is  defined as a pair $(f, \rho)$ consisting of 
\begin{itemize}
\item a local homeomorphism $f\col \ti{S} \to \CP^1$  ({\it developing map}) and 
\item a homomorphism  $\rho\col \pi_1(S) \to \PSL(2, \C)$ ({\it holonomy representation})
\end{itemize}
such that $f$ is $\rho$-equivariant (i.e. $f \alpha = \rho(\alpha) f$ for all $\alpha \in \pi_1(S)$). 
This pair$(f, \rho)$ is called the {\it developing pair} of $C$, and $(f, \rho)$ is, by definition,  equivalent to $(\gamma f,   \gam \rho \gam^{-1})$ for all $\gam \in \PSL(2, \C)$.
Due to  the equivariant condition, we do not usually need to distinguish between an element of $\pi_1(S)$ and its free homotopy class. 
Let $\PP$ be the deformation space of all $\CP^1$-structures on $S$; then  $\PP$  has a natural topology, given by  the open-compact topology on the developing maps $f\col \ti{S} \to \CP^1$. 

Notice that hyperbolic structures are, in particular, $\CP^1$-structures, as $\H^2$ is the upper half-plane in $\C$ and the orientation-preserving isometry group $\Isom \H^2$ is  the subgroup $\PSL(2, \R)$ of $\PSL(2, \C)$.\

\section{Grafting}\Label{sGrafting}
A grafting is a cut-and-paste operation of a $\CP^1$-structure inserting some structure along  a loop, an arc or more generally a lamination, originally due to \cite{Sullivan-Thurston-83, Hejhal-75, Maskit-69}. There are slightly different versions of grafting, but they all yield new $\CP^1$-structures without changing the topological types of the base surfaces.

A {\it round circle} in $\CP^1 = \C \cup \{\infty\}$ is a round circle in $\C$ or a straight line in $\C$ plus $\infty$. 
A {\it round disk} in $\CP^1$ is a disk bounded by a round circle.
An arc $\alpha$ on a $\CP^1$-structure is {\it circular} if $\alpha$ is sent into a round circle on $\CP^1$ by the developing map. 
Similarly, a loop $\alpha$ on a $\CP^1$-structure $C$ is {\it circular} if its lift $\ti\alpha$ to the universal cover is sent into a circular arc $\CP^1$ by the developing map. 

We first define a grafting along a circular arc on a $\CP^1$-structure.
For $\theta > 0$,  consider the horizontal biinfinite strip $\R \times [0, \theta i] $ in $\C$ of height $\theta$. 
Then let $R_\theta$ be the $\CP^1$-structure on the strip whose developing map is the restriction of the exponential map $\exp\col \C \to \C \minus \{0\}$.
This $\CP^1$-structure is called the {\it crescent} of angle $\theta$ or simply {\it $\theta$-crescent}. 

Let $\ell$ be a (biinfinite) circular arc properly embedded in a $\CP^1$-surface $C$. 
Then the {\it grafting} of $C$ along $\ell$ by $\theta$ is the insertion of this strip $R_\theta$ along $\ell$ \, ($\theta$-{\it grafting}), to be precise, as follows:
Notice that $C \minus \ell$ has two boundary components isomorphic to $\ell$. 
Then we take a union of  $C \minus \ell$ and $\R \times [0, \theta i]$ by  an isomorphism between $\bdr (C \minus \ell)$ and $\bdr(\R \times [0, \theta i] )$ so that there is ``no shearing'', i.e. for each $r \in \R$, the vertical arc $r \times [0, \theta i] $ connects the points of the different boundary components of  $C \minus \ell$ corresponding to the same point of $\ell$.

Let $\ell$ be a circular loop on a projective surface $C$. 
We can similarly define a grafting along $\ell$ by grafting the universal cover $\ti{C}$ of $C$ in an equivariant manner:
Letting $\phi\col \ti{C} \to C$ be the universal covering map, 
$\phi^{-1} (\ell)$ is a union of disjoint circular arcs property embedded in $\ti{C}$ which is invariant under $\pi_1(S)$.

Then, we insert  a $\theta$-crescent along each arc of $\phi^{-1}(\ell)$ as above. 
By quotienting out the resulting structure by $\pi_1(S)$, we obtain a new $\CP^1$-structure homeomorphic to $C$, since  a cylinder is inserted to $C$ along $\ell$. 
Indeed, the stabilizer of an arc $\ti\ell$ of $\phi^{-1}(\ell)$ is an infinite cyclic group generated by an element $\gamma \in \pi_1(S)$ whose free homotopy class is $\ell$, and the cyclic group $\langle \gamma \rangle$ acts on $R_\theta$ so that the quotient is the inserted cylinder ({\it grafting cylinder of height $\theta$}). 

Note that $R_\theta$ is foliated by horizontal lines $\R \times \{y\}$, $y \in [0, \theta]$.
Then it has a natural transverse measure given by the difference of the second coordinates. 
This measured foliation descends to a measured foliation on the grafting cylinder.
In addition, there is a natural projection $R_\theta \to \R$ to the first coordinate ({\it collapsing map}). 
This projection descends to a collapsing map of a grafting cylinder to a circle.

Let $\Gr_{\ell, \theta}(C)$ denote the resulting $\CP^1$-structure homeomorphic to $C$.  
Notice that the holonomy along the circular loop $\ell$  is hyperbolic, as it has exactly two fixed points on $\CP^1$ which are the endpoints of the developments of $\ell$. 

In the case that $\theta$ is an integer multiple of $2\pi$,  the holonomy $C$ is not changed by the $\theta$-grafting, since the developing map does not change in $\phi^{-1}(C \minus \ell)$. 
In particular, the $2\pi$-grafting along a circular loop $\ell$ inserts a copy of $\CP^1$ minus a circular arc along each lift of $\ell$. 

In fact, a $2\pi$-grafting is still well-defined along a more general loop. 
A loop $\ell$ on $C = (f, \rho)$ is ${\it admissible}$ if $\rho(\gam)$ is hyperbolic and an (equivalently, every) lift $\ti\ell$ of $\ell$ embeds into $\CP^1$ by $f$.  
Given such a loop, we can insert a copy of $\CP^1 \minus (f(\ti{\ell})  \cup {\rm Fix}(\rho(\gamma)))$ along $\ti\ell$, where ${\rm Fix}(\rho(\gam))$ denotes the fixed points of $\rho(\gam)$.
 Note that the quotient of  $\CP^1 \minus {\rm Fix}(\rho(\gam))$  by the infinite cyclic group generated by $\rho(\gamma)$ is a projective structure $T$ on a torus, and the development $f(\ti\ell)$ covers a simple loop  on $T$ isomorphic to $\ell$.
  By abuse of notation, we also denote the loop on $T$ by $\ell$.
Then the $2\pi$-grafting of $C$ along $\ell$ is given by identifying the boundary loops of $C \minus \ell$ and $T \minus \ell$ by the isomorphism. 
Denote by $\Gr_\ell(C)$  the $2\pi$-grafting of $C$ along an admissible loop $\ell$. 

A {\it multiloop} is a union of (locally) finite disjoint simple closed curves. 
Note that if there is a multiloop $M$  on a projective surface consisting of admissible loops, then a grafting can be done along $M$ simultaneously. 

\section{The construction of Thurston's parameters}\Label{ThurstonParameter}
In this section, we explain the correspondence stated in Theorem \ref{ThurstonCoordimates} in both directions, following \cite{Kullkani-Pinkall-94}. 
\subsection{The construction of $\CP^1$-structures from measured laminations on hyperbolic surfaces.}\Label{MLtoP}
Let  $(\tau, L ) \in \TT \times \ML$, where $\tau$ is a hyperbolic structure on $S$, and $L$ is a measured geodesic lamination on $\tau$. 
Then $(\tau, L)$ corresponds to  the $\CP^1$-structure on $S$ obtained by grafting $\tau$ along $L$ as follows.

Suppose first that $L$ consists of periodic leaves. 
Then, for each leaf $\ell$ of $L$, letting $w$ be its weight, we insert a grafting cylinder of height $w$, and obtain a projective structure $C = (f, \rho)$ on $S$. 
Let $\til{L}$ be the pull back of $L$ by the universal covering map.
Then there is a $\rho$-equivariant pleated surface $\beta\col\H^2 \to \H^3$, obtained by bending $\H^2$ along $\ti{L}$ by the angles given by the weights.  

Let $\kap\col C \to \tau$ be the collapsing map obtained by  collapsing all grafting cylinders in $C$ in \S \ref{sGrafting}. 
For each point $p$ in $\ti{C}$, there is an open neighborhood $D$, called a maximal disk, such that $f$ embeds $D$ onto a round disk in $\CP^1$.
Then, the boundary of $f(D)$ bounds a hyperbolic plane $H_p$ in $\H^3$.  
Denote, by $\Psi_p\col f(D) \to H_p$, the nearest projection. 
Then $\beta \circ\ti\kap (p) = \Psi_p \circ f(p)$ for all $p \in \ti{C}$, where $\ti\kap\col \ti{C} \to \H^2$ be the lift of $\kap\col C \to \tau$.

Suppose next that $L$ contains an irrational sublamination. 
Then,  pick a sequence of measured laminations $L_i$ consisting of closed leaves, such that $L_i$ converges to $L$ as $i \to \infi$. 
Then, for each $i$, as above there is a $\CP^1$-structure $C_i = Gr_{L_i} (\tau)$ and a $\rho_i$-equivariant pleated surface $\beta_i \col \H^2 \to \H^3$. 
As $L_i$ converges to $L$, then $\beta_i$ converges to a pleated surface $\beta\col \H^2 \to \H^3$ uniformly on compact sets, and therefore $C_i$ converges to a $\CP^1$-structure on $S$. 
(See \cite{CanaryEpsteinGreen84}.)

\subsection{The construction of measured laminations on hyperbolic surfaces from $\CP^1$-structures}
Let $C = (f, \rho)$ be a projective structure on $S$ given by a developing pair. 
Let $\ti{C}$ be the universal cover of $C$.

Identify $\CP^1$ conformally  with a unite sphere $\s^2$ in $\R^3$. 
Then, each round circle on $\CP^1$ is the intersection of $\s^2$ with some  (affine) hyperplane $\R^2$ in $\R^3$. 
A {\it (open) round disk} $D$ in $\ti{C}$ is an open subset of $\ti{C}$ homeomorphic to an open disk, such that $f$ embeds $D$ onto an open round disk in $\CP^1$ (we also say a maximal disk {\it of} $\ti{C}$, emphasizing the ambient space for the maximality). 
A {\it maximal disk} $D$ in $\ti{C}$ is a round disk, such that there is no round disk in $\ti{C}$ strictly containing $D$. 
Let $D$ be a maximal disk in $\ti{C}$. 
Then the closure of its image, $\overline{f(D)}$, is a closed round disk in $\CP^1$.

We first see a basic example illustrating the pleated surface corresponding to a $\CP^1$-structure. (See \cite{Epstein-Marden-87}.)
Let $U$ be a region of $\CP^1$ homeomorphic to an open disk such that $\CP^1 \minus U$ contains more than one point  (i.e. $U \ncong \CP^1, \C$).
Regard $\CP^1$ as the ideal boundary of hyperbolic three space $\H^3$, and consider the convex full of  $\CP^1 \minus U$  in $\H^3$.
Then its boundary in  $\H^3$ is a hyperbolic plane $\H^2$ bent along a measured lamination $L_U$ \cite{Epstein-Marden-87}.   
This lamination corresponds to the lamination in the Thurston coordinates. 

Let $\Psi_U$ denote the orthogonal projection from $U$ to  $\bdr \Conv (\CP^1 \minus U)$.  
Then, since $\bdr \Conv (\CP^1 \minus U)$ is, in the intrinsic metric,   a hyperbolic plane,   $\Psi$ yields a continuous map from $U$ to $\H^2$. 
For each maximal disk $D$ in $U$,  let $H_D$ be the hyperbolic plane in $\H^3$ bounded by its ideal boundary of $D$. 
Then $H_D$ intersects $\bdr \Conv (\CP^1 \minus U)$ in either a geodesic or the closure of a complementary region of $L_U$ in $\H^2$.  
Thus, all maximal disks in $U$ correspond to the {\it strata} of  $(\H^2, L)$, where
 each stratum is either the closure of a complementary region of $L$ in $\H^2$,  a leaf of $L$ with atomic measure,  or a leaf of $L$ not contained in the closure of some complementary region.
In particular, two distinct complementary regions $R_1, R_2$ of $(\H^2, L)$ correspond to different maximal disks $D_1, D_2$, and if $R_1$  is close enough to $R_2$, then $D_1$ intersects $D_2$.
Accordingly, the ideal boundary circles of $D_1$ and $D_2$ bound hyperbolic planes intersecting in a geodesic. 
Then the transverse measure of $L_U$ is, infinitesimally, given by the angles between such hyperbolic planes.  
 
Moreover there is a natural measured lamination $\LLL_U$ on $U$ which maps to $L_U$ by $\Psi_U$.
If a leaf $\ell$ has a positive atomic measure $w > 0$, then $\Psi_U^{-1}(\ell)$ is a crescent region $R_w$ of angle $w$, and $R_w$ is foliated by circular arcs $\ell'$ which project to $\ell$. 
Then  $\Psi_U$ is a homeomorphism in the complement of such foliated crescents, and  $\Psi_U$ isomorphically takes  $\LLL_U$ to $L_U$ in the complement  (i.e. it preserves leaves and transverse measure). 
The transverse measure of $\LLL$ is given by infinitesimal angles between ``very close" maximal disks. 
 
As developing maps of  $\CP^1$-structures are, in general,  not embedding, we need to find such projections somewhat more ``locally'' using maximal disks.

Let $D$ be a maximal disk in the universal cover $\ti{C}$, and let $\ol{D}$ be the closure of $D$ in $\ti{C}$. 
In other words, $\ol{D}$ is the connected component  of $f^{-1}(\overline{f(D)})$ containing $D$.
Then  $\ol{f(D)} \minus f(\ol{D})$ is a subset of the boundary circle of the round disk $f(D)$, and the points in this subset are called the {\it ideal points} of $D$. 
(Given a point $p$ of the boundary circle $f(D)$, pick a path $\alpha\col [0,1) \to f(D)$ limiting to $p$ as the parameter goes to $1$. 
Then $p$ is a ideal point of $D$ if and only if the lift of $\alpha$ to $\ti{C}$ leaves every compact subset of $\ti{C}$.) 

Let $\bdr_\infi D \subset \CP^1$ denote the set of all ideal points of $D$. 
As $f | D$ is an embedding onto a round disk, we regard $\bdr_\infi D$ as a subset of the boundary circle of $D$ abstractly (not as a subset of $\CP^1$). 
Then $\bdr_\infi D$ is a closed subset of $\mathbb{S}^1$, since its complement is open. 
Identifying $D$ with a hyperbolic disk conformally, we let $\Core(D)=  \Core_{\ti{C}}(D)$ be the convex hull of $\bdr_\infi D$.

For each point $p$ of $\ti{C}$, there is a round disk containing $p$, and moreover, as $C$ is not $\CP^1$ or $\C$, there is a maximal disk containing $p$. 
The {\it canonical neighborhood} $U_p$ of $C$ is the union of all maximal disks $D_j ~(j \in J)$ in $\ti{C}$ which contain $p$.

  In fact, $(C, \LLL)$ completely determines the Thurston parameters $(\tau, L)$. 
Furthermore  the Thurston parameters near $p \in \CP^1$ are determined by the Thurston parameters of its canonical neighborhood,  in the way similar to the  region $U$ in $\CP^1$ homeomorphic to a disk as above. 
Namely Lemma \ref{DisksInCanonicalNbhd} below implies that the Thurston lamination on $\ti{C}$ near $p$ is determined by the canonical neighborhood $U_p$ of $p$, and the following Proposition states that $U_p$ is a topological disk embedded in $\CP^1$. 
\begin{proposition}[\cite{Kullkani-Pinkall-94}, Proposition 4.1]\Label{CanonicalNeighborhood}
For every point $p$ in $\ti{C}$,
$f\col \ti{S} \to \CP^1$ embeds its canonical neighborhood $U_p$ into $\CP^1$. 
Moreover $U_p$ is homeomorphic to an open disk. 
\end{proposition}
\begin{proof}
Set $U_p = \cup D_j$, where $D_j$ are maximal disks in $\ti{C}$ containing $p$. 
Let $x, y$ be distinct points in  $U_p$; let $D_x$ and $D_y$ be maximal disks containing $\{p, x\}$ and $\{p, y\}$, respectively. 
By the definition of maximal disks,  $f$ embeds $D_i$ and $D_j$ onto round disks in $\CP^1$.
Then $f(D_i) \cap f(D_j) = f(D_i \cap D_j)$ is either a crescent  or a {\it round annulus}, i.e. a region in $\CP^1$ bounded by disjoint round circles. 
If it is a round annulus,  then $f| D_i \cup D_j$ must be a homeomorphism onto $\CP^1$ and $D_i \cup D_j = \til{S}$, which cannot occur. 
Thus $f(D_i \cap D_j)$ is a crescent, and therefore $f$ is injective on $D_x \cup D_y$.
Hence $f(x) \neq f(y)$, and  $f$ embeds $U_p$ into $\CP^1$. 

The image $f(U_p)$ is not surjective  (as $S$ is not homomorphic to a sphere). 
Thus we can  normalize $\CP^1 = \C \cup \{\infty\}$ so that $p = \infty$ and $0 \not\in U_p$.
Then $\CP^1 \minus  U_p$ is the intersection of the closed disks $\CP^1 \minus D_j$ containing $0$.
Thus $\CP^1 \minus U_p$ is a closed convex subset containing $0$, and therefore $U_p$ is topologically an open disk.
\end{proof}

Then in the setting of Proposition \ref{CanonicalNeighborhood}, we have
\begin{corollary}\Label{Complement}
When  $\CP^1 = \C \cup \{\infty\}$ is normalized so that $p = \{\infty\}$, 
the complement $\CP^1 \minus U_p$ is a compact convex subset $K$ of $\C$.
\end{corollary}

The canonical neighborhood $U_p$ can be regarded  as a projective structure on an open disk (Proposition \ref{CanonicalNeighborhood}), and one can consider maximal disks in $U_p$, which are {\it a priori}
 unrelated maximal disks in $\ti{C}$. 

\begin{lemma}\Label{DisksInCanonicalNbhd}
The maximal disks of $U_p$ bijectively correspond to the maximal disks of $\ti{C}$ whose closure contain $p$ by the inclusion $U_p \sub \ti{C}$.

Moreover, if $D$ is a maximal disk of $U_p$ containing $p$, then the ideal points of $D$ as a maximal disk of $U_p$ coincide with its ideal points of $D$ as a maximal disk of $\ti{C}$.
\end{lemma}
\begin{proof}
If $D$ is a maximal disk in $\ti{C}$ containing $p$, then clearly $D$ is also a maximal disk in $U_p$ by the definition of $U_p$.  
Similarly, if $D$ is a maximal disk in $\ti{C}$ whose boundary contains $p$, then there is a sequence of maximal disks $D_i$ containing $p$ with $D_i \to D$ as $i \to \infi$.
Therefore every maximal disk $D$ in $\ti{C}$ whose closure contains $p$ is a maximal disk in $U_p$. 

We show the opposite inclusion.
By Corollary \ref{Complement}, 
the complement $K = \CP^1 \minus U_p$ is a closed compact convex subset of $C$. 
If there is a (straight) line $\ell$ in $\C$ such that $\ell \cap K$ is a single point $x$, then, by the inclusions $\ti{C} \supset U_p \sub \CP^1$, $x$  corresponds to an ideal point of a maximal ball of $\ti{C}$ containing $p$.
Next suppose that there is a line $\ell$ in $\C$ such that $\ell \cap K$ is a line segment. 
 Then,  letting $P$ be the half-plane bounded by $\ell$ so that $P$ and $K$ have disjoint interiors, there is a sequence of maximal disks $D_i$ of $\ti{C}$ containing $p$ such that $D_i$ converges to $P$ as $i \to \infi$.  
Thus the endpoints of the line segment correspond to ideal points of $\ti{C}$. 

Suppose that $D$ is a maximal disk of $U_p$.
Then $\ol{D}$ intersects $K$ in $\bdr K$. 
If $\bdr D$ intersects $K$ in a line segment, then $D$ is a half-plane in $\C$ with  $\bdr D$ containing $p$. 
As the endpoints of the segment correspond to the ideal points of $\ti{C}$, $D$ is also a maximal disk in $\ti{C}$.

If the closure of $D$ does not intersect $K$ in a line segment, then clearly $D$ contains $p$.
If a point on $\bdr K$ is not an interior point of any line segment of $\bdr K$, then the point corresponds to an ideal point of $\ti{C}$.
Therefore  no round disk in $\ti{C}$ strictly contains $D$, and therefore $D$ is also a maximal ball in $\ti{C}$.
Thus we have the opposite inclusion. 

Finally, suppose that $D$ is a maximal ball in $U_p$ containing  $p$. 
Then $\ol{D} \cap K$ contains no line segment, and therefore, $\ol{D} \cap K$ corresponds to the ideal points of $D$ as a maximal ball in $\ti{C}$.  
\end{proof}

The following proposition yields a lamination on $\ti{C}$ invariant under $\pi_1(S)$.
\begin{proposition}[\cite{Kullkani-Pinkall-94}, Theorem 4.4]\Label{stratification}
The cores $\Core(D)$ of the maximal disks $D$ in $\ti{C}$ are all disjoint and their union is $\ti{C}$.
\end{proposition}
\proof
We first show that the cores are disjoint.
Let $D_1$ and $D_2$ be distinct maximal disks in $\ti{C}$. 
If $D_1 \cap D_2 \neq \emptyset$, then $f(D_1)$ and $f(D_2)$ are round disks intersecting a crescent.  
Therefore $\Core(D_1) $ and $\Core(D_2)$ are disjoint. (Consider the circular arc in $D_1$ orthogonal to $\bdr D_1$; then, indeed, this arc separates $\Core(D_1)$ and $\Core(D_2)$ in $D_1 \cup D_2$.)

\begin{claim}\Label{MinimalRadius}
Given a convex subset $V$ of $\C$, there is a unique round disk $D$ in $\C$ of minimal radius containing $V$. 
\end{claim}
\begin{proof}
Suppose, to the contrary,  that there are two different round disks $D_1, D_2$ containing $V$ which attain the minimal radius.  
Then, clearly, there is a round disk $D_3$ of strictly smaller radius which contains  $V$ (such that $D_3 \supset D_1 \cap D_2$  and $D_3 \sub D_1 \cup D_2$).
This is a contradiction. 
\end{proof}

\begin{claim}\Label{CenterInHull}
The convex hull of  $\bdr D \cap \ol{V}$ contains the center of $c$ with respect to the complete hyperbolic metric on $\,D (\cong \H^2)$ given by the conformal identification.
\end{claim}
\begin{proof}
Suppose not; then the closure of $V$ is contained in the interior of a (Euclidean) half disk of $D$.
Then one can easily find a round disk of smaller radius containing $\ol{V}$, 
\end{proof}
Note that, by the inversion of $\CP^1 = \C \cup \{\infty\}$ about $\bdr D$ exchanges $\infty$ and the center of $D$, and it fixes $\bdr D$ pointwise. 
Then, by Claim \ref{CenterInHull},  in the (conformal) hyperbolic metric on $D$,  the convex hull of $\bdr D \cap \ol{V}$ contains the center of $D$.
Therefore, by the inversion, in the hyperbolic metric on $\CP^1 \minus D$,    the point at $\infty$ is contained in the convex hull of $\bdr D \cap \ol{V}$ in the interior of $\CP^1 \minus D$.

Using the above claims,  we show that, for every $x \in \ti{C}$, there is a maximal disk $D$ in $\ti{C}$ whose core contains $x$. 
Let $U_x = \cup_{j \in J} D_j$ be the canonical neighborhood of $x$, where $D_j$ are the maximal disks in $\ti{C}$ which contain $x$.
Normalize $\CP^1$ so that $f(x) = \infty$. 
Let $D_j^c = \CP^1 \minus f(D_j)$.
Then $\CP^1 \minus f(U_x) = \cap_j D_j^c$.
By Claim \ref{CenterInHull}, let $D$ be the maximal disk of $U_x$  such that $x  \in \Core_{U_x}(D)$.
By Lemma \ref{DisksInCanonicalNbhd}, $D$ is also a maximal disk of $\ti{C}$ which contains $x$, and moreover the ideal points of $D$ as a maximal of $U_x$ coincide with those  as a maximal ball of $\ti{C}$.
Then,  $\Core_{\ti{C}}(D)$ contains $x$.
\Qed{stratification}

By Proposition \ref{stratification}, $\til{C}$ is canonically decomposed into the cores of maximal disks in $\ti{C}$, which yields a  {\it stratification} of $\ti{C}$.
Note that this decomposition is invariant under $\pi_1(S)$, as the maximal balls and ideal points are preserved by the action. 
Moreover, for each maximal disk $D$ in $\ti{C}$, its $\Core(D)$ is properly embedded in $\ti{C}$.
Then the one-dimensional cores and the boundaries components of two-dimensional cores form a $\pi_1(S)$-invariant lamination $\ti\lambda$ on $\ti{C}$, which descends to a lamination $\lambda$ on $C$.

Next we see that the angles between infinitesimally close maximal disks yield a natural transverse measure of this lamination.
Given a point $x \in \ti{C}$, let $D_x$ be the maximal disk in $\ti{C}$ whose core contains $x$.
If $y \in \ti{C}$ is sufficiently close to $x$, then $D_y$ intersects $D_x$. 
Then let $\angle(D_x, D_y)$ denote the angle between the boundary circles of $D_x$ and $D_y$.  
To be precise, it is the angle of  the crescent $D_x \minus D_y$ (or $D_y \minus D_x$) at the vertices. 
 Then $\angle(D_x, D_y) \to 0$ as $y \to x$.  

Let $x$ and $y$ be distinct points of $\ti{C}$ contained in different strata of $(\til{C}, \ti\lambda)$.
Then pick a path $\alpha\col [0,1] \to \ti{C}$ connecting $x$ to $y$ such that $\alpha$ is transverse to $\ti\lambda$. 
Let $\Delta: 0 = t_0 < t_1 < \dots < t_n = 1$ be a finite division of $[0,1]$, and let $x_i = \alpha(t_i)$ for each $i = 0, \dots, n$. 
Let $|\Delta| = \min_{i = 0}^{n-1} (x_{i + 1} - x_i)$, the smallest width of the subintervals.
Then, let  $\Theta(\Delta) = \Sigma_{i = 1}^{n-1} \angle(D_{x_i}, D_{x_{i + 1}})$ for a subdivision $\Delta$ of $[0,1]$ with sufficiently small $| \Delta |$. 
Pick a sequence of subdivisions $\Delta_i$ such that  $| \Delta_i| \to 0$ as $i \to \infi$.
Then $\lim_{i \to \infi} (\Theta(\Delta_i))$ exists and it is independent on the choice of $\Delta_i$ as $i \to \infi$ (\cite[II.1] {CanaryEpsteinGreen84}).
We define the transverse measure of $\alpha$ to be  $\lim_{i \to \infi} (\Theta(\Delta_i))$.
Then $\ti\lambda$ with this transverse measure yields a measured lamination  $\ti\LLL$ invariant under $\pi_1(S)$.
Thus $\ti\LLL$ descends to a measured lamination $\LLL$ on $C$.

 By Lemma \ref{DisksInCanonicalNbhd}, for every $x \in \ti{C}$,  the measured lamination $\LLL$ near $x$ is determined by the canonical neighborhood $U_x$ of $x$.
Let $\LLL_x$ be the measured lamination on $U_x$, which descends to the measured lamination on the boundary of $\Conv (\CP^1 \minus U_x)$.
Then there is a neighborhood $V$ of $x$ in $U_x$ such that $\LLL$ and $\LLL_x$ coincide in $V$ by the inclusion $U_x \sub \ti{C}$. 

For each point $x \in \til{C}$, the boundary circle of the maximal disk $D_x$ bounds a hyperbolic plane $H_x$ in $\H^3$. 
Let $\Psi_x \col f(D_x) \to H_x$ be the projection along geodesics in $\H^3$ orthogonal to $H_x$. 
Then $H_x$ has a canonical normal direction pointing to $D_x$.
By Lemma \ref{DisksInCanonicalNbhd} there is a neighborhood $V$ of $x$, such that $\Psi_y(y) = \Psi_x(y)$.
Moreover, $\Psi_x$ coincides with the projection onto the boundary pleated surface of $\Conv \CP^1 \minus U_x$.
Therefore, as in the case of regions in $\CP^1$, we have a pleated surface $\H^2 \to \H^3$ which is $\rho$-equivariant, as in the following paragraph.

We assume that crescents $R$ in $\ti{C}$ are always foliated by leaves of $\ti{\LLL}$ sharing their endpoints at the vertices of $R$.
  We have a well-defined continuous map $\Psi \col \ti{C} \to \H^3$ defined by $\Psi(x) = \Psi_x(x)$.
We shall take an appropriate quotient of $\ti{C}$ to turn it into  a hyperbolic plane. 
For each crescent $R$ in $\ti{C}$, $\Psi$ takes each leaf in $R$ to the geodesic in $\H^3$ connecting the vertices of $R$.
Identify $x, y \in \ti{C}$, if $x, y$ are contained in a single crescent in $\til{C}$  and $\Psi_x(x) = \Psi_y(y)$; let $\ti\kappa\col \ti{C} \to \ti{C}/\sim$ be the quotient map by this identification, which  collapses each foliated crescent region to a single leaf.
Then by the equivalence relation, $\Psi \col \ti{C} \to \H^3$ induces a continuous map $\beta\col (\ti{C}/ \sim) \to \H^3$ such that $\Psi_x(x) = \beta\cc \kap$. 
Moreover, $\ti{C}/\sim$ is $\H^2$ with respect to the path metric in $\H^3$ via $\Psi$, since, for every $x \in \ti{C}$, $\Psi$ coincides with the projection $U_x \to \bdr \Conv(\CP^1 \minus U_x)$ in a neighborhood of $x$.
Thus we have a $\rho$-equivariant pleated surface $\H^2 \to \H^3$.

The measured lamination $\til{\LLL}$ on $\ti{C}$ descends to a measured lamination $\ti{L}$ on $\H^2$ invariant under $\pi_1(S)$.  
By taking the quotient, we obtain a desired pair $(\tau, L)$ of  a hyperbolic surface $\tau$ and a measured geodesic lamination $L$ on $\tau$. 

Similarly, the collapsing map $\ti\kap\col \ti{C} \to \H^2$ descends to a {\it collapsing map} $\kap\col C \to \tau$. 
Then, for each periodic leaf $\ell$ of $L$, $\kap^{-1}(\ell)$ is a grafting cylinder foliated by closed leaves of $\LLL$. 

Finally we note that as $\beta\col\H^2 \to \H^3$ is obtained by bending $\H^2$ in $\H^3$ along $\ti{L}$, the pair $(\tau, L)$ corresponds to $C$  by the correspondence in \S\ref{MLtoP}.

\section{Goldman's theorem on projective structures with Fuchsian holonomy}\Label{ProofGoldman}

Let $C$ be a $\CP^1$-structure on $S$ with holonomy $\rho$, and let $(\tau, L) \in \TT \times \ML$ be its  Thurston parameters.
Let $\psi\col \H^2 \to \tau$ be the universal covering map, and $\ti{L}$ be the measured lamination $\psi^{-1} (L)$ on $\H^2$.
Let $\Gam = \Im \rho$, and let $\Lambda$ denote the limit set of $\Im \rho$.

\begin{lemma}\Label{PleatintAtLimitSet}
Let $\beta\col \H^2 \to \H^3$ be the associated pleated surface, where $\H^2$ is the universal cover of $\tau$. 
Then, for every leaf $\ti\ell$ of $\til{L}$, $\beta | \ti\ell$ is a geodesic connecting different points of $\Lambda$.
\end{lemma}

\begin{proof}
If $\ti\ell$ is a lift of a closed leaf of $L$, then the assertion  clearly holds.

For every closed curve  $\alpha$ on $\tau$, let $\ti\alpha$ be a lift of $\alpha$ to $\H^2$.
Since the curve $\beta | \til{\alp}$ is preserved by the hyperbolic element   $\rho(\alpha)$,
 it is a quasi-geodesic in $\H^3$ whose endpoints are the fixed points of $\rho(\alp)$. 
 Note that the endpoints are contained in $\Lam$.

Let $\ell$ be a non-periodic leaf of $L$, and  let $\ti\ell$ be a lift of $\ell$ to $\H^2$. 
There is a sequence of simple closed geodesics $\ell_i$ on $\tau$ such that $\ell_i$ converges to $\ell$ in the Hausdorff topology  (\cite[I.4.2.14]{CanaryEpsteinGreen84}).
For each $i \in \N$, pick a lift $\til{\ell}_i$ of $\ell_i$ to $\H^2$ so that $\til{\ell}_i \to \ell$  uniformly on compact sets as $i \to \infi$. 
Then, $\beta | \til{\ell}_i$ converges to $\beta | \til{\ell}$ uniformly on compact sets.
Moreover as $\angle_{\tau_i}(\tau_i, L_i) \to 0$,   $\beta_i | \til{\ell}_i$ is asymptotically an isometric embedding: To be precise,  for large enough $i$,  it is a bilipschitz embedding, and its bilipschitz constant converges to $1$ as $i \to \infi$ \cite[Proposition 4.1]{Baba-15gt}. 

As $\ell_i$ are closed loops, the endpoints of $\beta | \til{\ell}_i$ are in $\Lambda$.
Then the endpoints of  $\beta | \til{\ell}_i$ converge to the endpoints of $\beta | \til{\ell}$ in $\CP^1$.
Therefore,  since $\Lambda$ is a closed subset of $\bdr \H^3$, the endpoints of $\beta | \ell$ are also contained in $\Lambda$. 
\end{proof}
We immediately have
\begin{corollary}\Label{FuchsianStrata} 
For each stratum $\sigma$ of $(\H^2, \til{L})$, let $D_\sigma \sub \ti{C}$ be the maximal disk whose core corresponds to $\sigma$. 
Then its ideal points $\bdr_\infi D_\sigma$ are contained in the limit set $\Lambda$.
\end{corollary}

We reprove the following theorem by means of pleated surfaces.
\begin{proposition}(See \cite[Theorem 3.7.3.]{Tan-88})\Label{TwoPiMultiple}
Let $C$ be a $\CP^1$-structure with real holonomy $\rho\col \pi_1(S) \to \PSL(2, \R)$ and 
  $(L, \tau)$ its Thurston parameters. 
Then each leaf of $L$ is periodic, and its weight is $\pi$-multiple. 
If $\rho$ is, in addition, Fuchsian,  then each leaf of $L$ is periodic and its weight is a $2\pi$-multiple. 
\end{proposition}
\begin{proof}

We first show that $L$ consists of periodic leaves. 
Suppose, to the contrary,  that $L$ contains an irrational minimal sublamination $N$. 
Then the transverse measure is continuous in a neighborhood of  $| N | $ in $\tau$ (i.e. no leaf of $N$ has an atomic measure). 

Thus there are two-dimensional strata  $\sigma, \sigma_1, \sigma_2, \dots$ of $\H^2 \minus \ti{L}$, such that $\sigma_i$ converges to an edge of $\sigma$ as $i \to \infi$. 
Note that, as they are two-dimensional, each $\beta(\sigma_i)$ has at least three ideal points, which lie in a round circle in $\CP^1$.
Let $H, H_1, H_2, \dots$ be the supporting oriented hyperbolic planes in $\H^3$ of $\sigma, \sigma_1, \dots$. 
Let $\angle_{\H^3}(H, H_i) \in [0, \pi]$ be the angle between the hyperbolic planes $H$ and $H_i$ with respect to their orientations, if $H$ and $H_i$ intersect.
Then, by continuity, $\angle_{\H^3}(H, H_i)  \to 0$  as $i \to \infi$.
Thus the ideal points of $\sigma$ and $\sigma_i$  cannot be contained in a single round circle if $i$ is sufficiently large. 
By Corollary \ref{FuchsianStrata}, this cannot happen  as $\Lambda$ is a single  round circle. 

We first show that the weight of each leaf of $L$ is a multiple of $\pi$. 
Let $\sigma_1$ and $\sigma_2$ be components of $\H^2 \minus \ti{L}$ adjacent along a leaf of $\ti{L}$. 
Let $H_1$ and $H_2$ be the support planes of $\sigma_1$ and $\sigma_2$, respectively.
Then the angle between $H_1$ and $H_2$ is the weight of $\ell$.
As the ideal points of $\sigma_1$ and $\sigma_2$ must lie in the round circle $\Sigma$, the angle must be a multiple of $\pi$. 

Suppose, in addition, that $\rho$ is Fuchsian. 
Let $\beta_0\col \H^2 (= \ti\tau ) \to \H^3$ be the $\rho$-equivariant embedding onto the hyperbolic plane  $H_\Lambda$ bounded by $\Lambda$. 
For each $i = 1,2$,  as each boundary component $m$ of $\sigma_i$ covers a periodic leaf of $L$,  $\beta = \beta_0$ on $m$. 
 Therefore $H_1 = H_2 =  \Conv(\Lambda)$, and $\beta_0 = \beta$ on $\sigma_i$ for each $i = 1, 2$.
 As the orientation of $H_1$ coincides with that of $H_2$, the weight of $m$ must be a multiple of $2\pi$. 
\end{proof}

\proof[Proof of Theorem \ref{Goldman}]
By Proposition,
\ref{TwoPiMultiple}, $L$ is a union of closed geodesics $\ell$ with $2\pi$-multiple weights.
For each (closed) leaf $\ell$ of $L$,  let $2 \pi n_\ell$ denote the weight of $\ell$, where $n_\ell$ is a positive integer.
Let $\kap\col C \to \tau$ be the collapsing map. 
Then, $\kap^{-1}(\ell)$  is a grafting cylinder of height $2\pi n_\ell$, the structure inserted by $2\pi$-grafting $n$ times.
Therefore, $C$ is obtained by grafting along a multiloop corresponding to $L$.
\endproof

\section{The path lifting property in the domain of discontinuity }\Label{LfitingProperty}
Let $C = (f, \rho)$ be a $\CP^1$-structure on $S$.
Then, let $\Lambda$ be the limit set of $\Im \rho$, and let $\Omega = \CP^1 \minus \Lam$, the domain of discontinuity.  

\begin{proposition}
For every $x \in \Omega$, there is  a neighborhood $V_x$ in  $\Omega$ such that,  for every $y \in \til{S}$ with $f(y) \in V_x$,   $V_x$ is contained in the maximal disk whose core contains $x$.
\end{proposition}

\begin{proof}
The union $\H^3 \cup \bdr \H^3$ is a  unit ball in the Euclidean space and  the visual distance is the restriction of the Euclidean metric. 

Suppose, to the contrary, that there is no such neighborhood $V_x$. 
Then there is a sequence $x_1, x_2, \dots \in f^{-1}(x)$ such that, letting $H_1, H_2, \dots$ be their corresponding hyperbolic support planes, 
the visual distance from $H_i$ to $x$ goes to zero as $i \to \infi.$
Let $y_i \in \H^3$ be the nearest point projection of $f(x_i)$ to $H_i$.
Then,  $y_i \to x$ in the visual metric. 
Let $\sigma_i$ be the stratum of $(\H^2, \ti{L}_i)$ which contains $\ti\kappa(x_i)$. 
Then, as the orthogonal projection of $f(x_i)$  to $H_i$ is $y_i$, the visual distance between $x$ and  $\beta_i(\sigma_i)$ goes to zero as $i \to \infi$.
Therefore, there is an ideal point $p_i$ of $\beta(\sigma_i)$ which converges to $x$ as $i \to \infi$.
As $\Omega$ is open,  this is a contradiction by Corollary \ref{FuchsianStrata}.
\end{proof}

As $f$ embeds maximal disks of $\ti{C}$ into $\CP^1$, we immediately have 

\begin{corollary}
For each point $x \in \Omega$, there is a neighborhood $V_x$ of $x$ such that, if $f(y) \in V_x$ for $y \in \til{S}$, then $f$ embeds a neighborhood $W_y$ of $y$ in $\ti{S}$ homomorphically onto $V_x$. 
\end{corollary}
 Theorem \ref{Covering} immediately follows from the corollary.

\bibliographystyle{abbrv}
\bibliography{ThurstonProjectiveStructures}

\end{document}